\documentclass{amsart}

\usepackage[utf8]{inputenc}
\usepackage{amsmath,amsthm,amssymb}
\usepackage{graphicx,xcolor}
\usepackage{url}
\usepackage[algo2e, ruled, linesnumbered]{algorithm2e}
\usepackage{booktabs}
\usepackage{multirow}

\newtheorem{theorem}{Theorem}
\newtheorem{lemma}[theorem]{Lemma}

\newtheorem{definition}[theorem]{Definition}

 % 90C20?

\date{\today}

\newcommand{\R}{\mathbb{R}}
\newcommand{\T}{\mathrm{T}}

\DeclareMathOperator*{\maximize}{maximize}
\DeclareMathOperator*{\st}{subject\;to}
\DeclareMathOperator{\bd}{bd}

\newcommand{\vA}{\mathbf{A}}
\newcommand{\vb}{\mathbf{b}}
\newcommand{\vc}{\mathbf{c}}
\newcommand{\vI}{\mathbf{I}}

\newcommand{\vQ}{\mathbf{Q}}
\newcommand{\vr}{\mathbf{r}}

\newcommand{\vx}{\mathbf{x}}
\newcommand{\vy}{\mathbf{y}}
\newcommand{\vzero}{\mathbf{0}}

\newcommand{\deletethis}[1]{{}}

\numberwithin{equation}{section}

\begin{document}

\begin{center}
\large{\bf{The smallest mono-unstable, homogeneous convex polyhedron has at least 7 vertices}}  \\[5mm]

\normalsize
S{\'a}ndor Boz{\'o}ki \\[2mm]
corresponding author \\ 
{HUN-REN Institute for Computer Science and Control (SZTAKI), Kende street 13-17, Budapest, Hungary, 1111; \\
Corvinus University of Budapest} \\ 
\texttt{bozoki.sandor@sztaki.hun-ren.hu} \\[5mm]
% ORCID

G{\'a}bor Domokos \\[2mm]
{Department of Morphology and Geometric Modeling and MTA-BME Morphodynamics Research Group, Budapest University of Technology and Economics, Budapest, Hungary} \\
 M\H uegyetem rakpart 1-3., Budapest, Hungary, 1111 \\
\texttt{domokos@iit.bme.hu}  \\[5mm]
% ORCID 0000-0002-8676-6829

D{\'a}vid Papp \\[2mm]
{Department of Mathematics, North Carolina State University, Raleigh, NC, USA. \\ \texttt{https://orcid.org/0000-0003-4498-6417}} \\
Box 8205, NC State University, Raleigh, NC 27695-8205, USA  \\
\texttt{dpapp@ncsu.edu} \\[5mm]
% ORCID

Krisztina Reg{\H o}s \\[2mm]
{Department of Morphology and Geometric Modeling and MTA-BME Morphodynamics Research Group, Budapest University of Technology and Economics, Budapest, Hungary} \\
 M\H uegyetem rakpart 1-3., Budapest, Hungary, 1111  \\
\texttt{regos.kriszti@gmail.com} \\[9mm]
% ORCID 0000-0001-6866-2658)
\end{center}

{DP: Research was done while visiting the Corvinus Institute of Advanced Studies, Corvinus University, Budapest, Hungary. This material is based upon work supported by the National Science Foundation under Grant No.~DMS-1847865. This material is based upon work supported by the Air Force Office of Scientific Research under award number FA9550-23-1-0370. Any opinions, findings and conclusions or recommendations expressed in this material are those of the author(s) and do not necessarily reflect the views of the U.S. Department
of Defense. GD: Support of the NKFIH Hungarian Research Fund grant 134199 and of grant BME FIKP-VÍZ by EMMI is kindly acknowledged. SB: The research has been supported in part by the TKP2021-NKTA-01 NRDIO grant. KR: This research has been supported by the program UNKP-23-3 funded by ITM and NKFI. The research was also supported by the Doctoral Excellence Fellowship Programme (DCEP) funded by ITM and NKFI and the Budapest University of Technology and Economics. The gift representing the Albrecht Science Fellowship is gratefully appreciated. }

\newpage

\section*{Abstract}
%\begin{abstract}
We prove that every homogeneous convex polyhedron with only one unstable equilibrium (known as a mono-unstable convex polyhedron) has at least $7$ vertices. Although it has been long known that no mono-unstable tetrahedra exist, and mono-unstable polyhedra with as few as $18$ vertices and faces have been constructed, this is the first nontrivial lower bound on the number of vertices for a mono-unstable polyhedron.

There are two main ingredients in the proof. We first establish two types of relationships, both expressible as (non-convex) quadratic inequalities, that the coordinates of the vertices of a mono-unstable convex polyhedron must satisfy, taking into account the combinatorial structure of the polyhedron. Then we use numerical semidefinite optimization algorithms to compute easily and independently verifiable, rigorous certificates that the resulting systems of quadratic inequalities (5943 in total) are indeed inconsistent in each case.
%\end{abstract}

\section{Introduction}

The stability of convex bodies, specifically the number of stable and unstable equilibria that a convex body (and, in particular, a convex polyhedron) may have, is a fundamental area of study in geometry. It is
intuitively clear that every convex body has at least $S=1$ stable equilibrium (a point on its boundary on which it can stand on a horizontal surface without toppling over) and at least $U=1$ unstable equilibrium (a point on its boundary on which it can theoretically be balanced like a pencil on its tip)---the points of the boundary closest and farthest from the center of mass always qualify as a stable and an unstable equilibrium, respectively. For general, $C^1$-smooth convex bodies Arnold conjectured \cite{VarkonyiDomokos} that  that $S+U=2$ can be realized as a homogeneous object  and a physical example has also been constructed \cite{VarkonyiDomokos}. In the case of polyhedra, one may wonder what is the minimally necessary number of faces and vertices to achieve either $S=1$ (a \emph{mono-stable} polyhedron), $U=1$ (a \emph{mono-unstable} polyhedron) or both. The first homogeneous convex polyhedron with $S=1$ was constructed by Conway and Guy \cite{ConwayGuy1969} in 1969 with $F=19$ faces and $V=34$ vertices; this was followed by other constructions with slightly fewer faces and vertices \cite{Bezdek, Reshetov}; similarly, homogeneous mono-unstable polyhedra have been constructed in \cite{balancing}. All these constructions provide \emph{upper} bounds for the necessary numbers of faces and vertices. On the other hand, surprisingly little is known about the lower bound: Conway proved, using an elementary argument that homogeneous tetrahedra have $S>1$ \cite{Dawson2} and using polar duality, $U>1$ was also proven for tetrahedra \cite{balancing}. %In this paper we provide the first nontrivial lower bound by proving that $U>1$ follows from $V<7$:
The main contribution in this paper is the first nontrivial lower bound on the number of vertices of a mono-unstable polyhedron.
\begin{theorem}
\label{thm:main}
Every homogeneous mono-unstable polyhedron has at least 7 vertices.
\end{theorem}

The approach is an adaptation of the authors' earlier work on the stability of $0$-skeleta \cite{Bozoki, PappEJOR} to this significantly more difficult case of \emph{homogeneous} polyhedra. We first produce various systems of polynomial inequalities with the property that if none of them has a solution, then homogeneous mono-unstable polyhedra with $V=5$ or $6$ vertices do not exist. Then we employ semidefinite optimization to show that the polynomial systems are indeed inconsistent (or ``infeasible'', in optimization terminology). The proof is thus computer-assisted and relies on inexact numerical computation with in principle unverifiable software, but the automated portion of the work produces \emph{easily and independently verifiable, rigorous certificates} that the polynomial systems are indeed infeasible that can be validated independently of how they were generated. This is not an entirely new concept in geometric proofs: similar rigorous computer-generated proofs are given, for example, by Bachoc and Vallentin \cite{BachocVallentin2008} of their lower bounds on the ``kissing problem'' (the maximum number of non-intersecting unit spheres touching a fixed unit sphere in $n$ dimensions); another example is Cohn and Woo's lower bounds on energy-minimizing point configurations on spheres \cite{CohnWoo2012}.

\subsection{Stability of convex polyhedra}
\label{sec:stability}

Static equilibra can be understood in mechanical terms: a static balance point of a convex body is a point on its surface on which the body can rest if supported on a horizontal plane. For our analysis the following, purely geometric, definitions of static equilibria specific to convex polyhedra are useful. Although we focus on $3$-dimensional polyhedra in this paper, the approach is applicable in any dimension, thus we shall give our fundamental definitions for the general case.

\begin{definition}\label{def:equilibria}
Let $P \subseteq \R^d$ be a $d$-dimensional convex polyhedron, let $\bd P$ denote its boundary, and let $c$ be the center of mass of $P$.  We say that $q \in \bd P$ is an \emph{equilibrium point} of $P$ with respect to $c$ if the hyperplane $h$ through $q$ and perpendicular to the line segment $[c,q]$ supports $P$ at $q$. In this case $q$ is \emph{nondegenerate} if $h \cap P$ is the (unique) $k$-dimensional face ($k=0,1, \dots d-1)$ of $P$ that contains $q$ in its relative interior. A nondegenerate equilibrium point $q$ is called \emph{stable} or \emph{unstable}, if $\dim (h \cap P) = d-1$, or $0$, respectively, otherwise we call it a \emph{saddle-type} equilibrium. %We denote the respective numbers of stable and unstable equilibria by $S$ and $U$.
\end{definition}

Unstable equilibria of convex polyhedra are thus associated with vertices:

\begin{definition}\label{def:unstable}
A vertex $v$ of a convex polyhedron $P$ is a \emph{nondegenerate unstable vertex} of $P$ if the hyperplane that contains $v$ and is orthogonal to the line connecting the center of mass and $v$ (called the \emph{vertex orthoplane} of $P$ at $v$) intersects $P$ only at $v$.
\end{definition}

Intuitively, the polyhedron can stand on any of its unstable vertices on a horizontal plane with its center of mass (vertically) above the vertex; see also \mbox{Figure \ref{fig:unstable}}.

\section{From mono-instability to systems of polynomial inequalities}

If a vertex is \emph{not} unstable, this can always be attributed to its position relative to the center of mass and just a single adjacent vertex. The corresponding relation was introduced in \cite{Bozoki}:
\begin{definition}
Let $i$ and $j$ be two vertices of a polyhedron with coordinate vectors $\vr_i$ and $\vr_j$, and let $\vc$ be the coordinate vector of the center of mass of the polyhedron. We say that $j$ \emph{shadows} $i$ (or $i$ is \emph{in the shadow of} $j$) if
\begin{equation}\label{eq:shadow}
(\vr_i-\vr_j)^\T(\vr_i - \vc) \leq 0.
\end{equation}
\end{definition}

The next lemma reveals how the shadowing relationship between adjacent vertices can be used to identify unstable vertices. See also Figure \ref{fig:unstable}.

\begin{figure}
    \centering
    \includegraphics[width=0.6\linewidth]{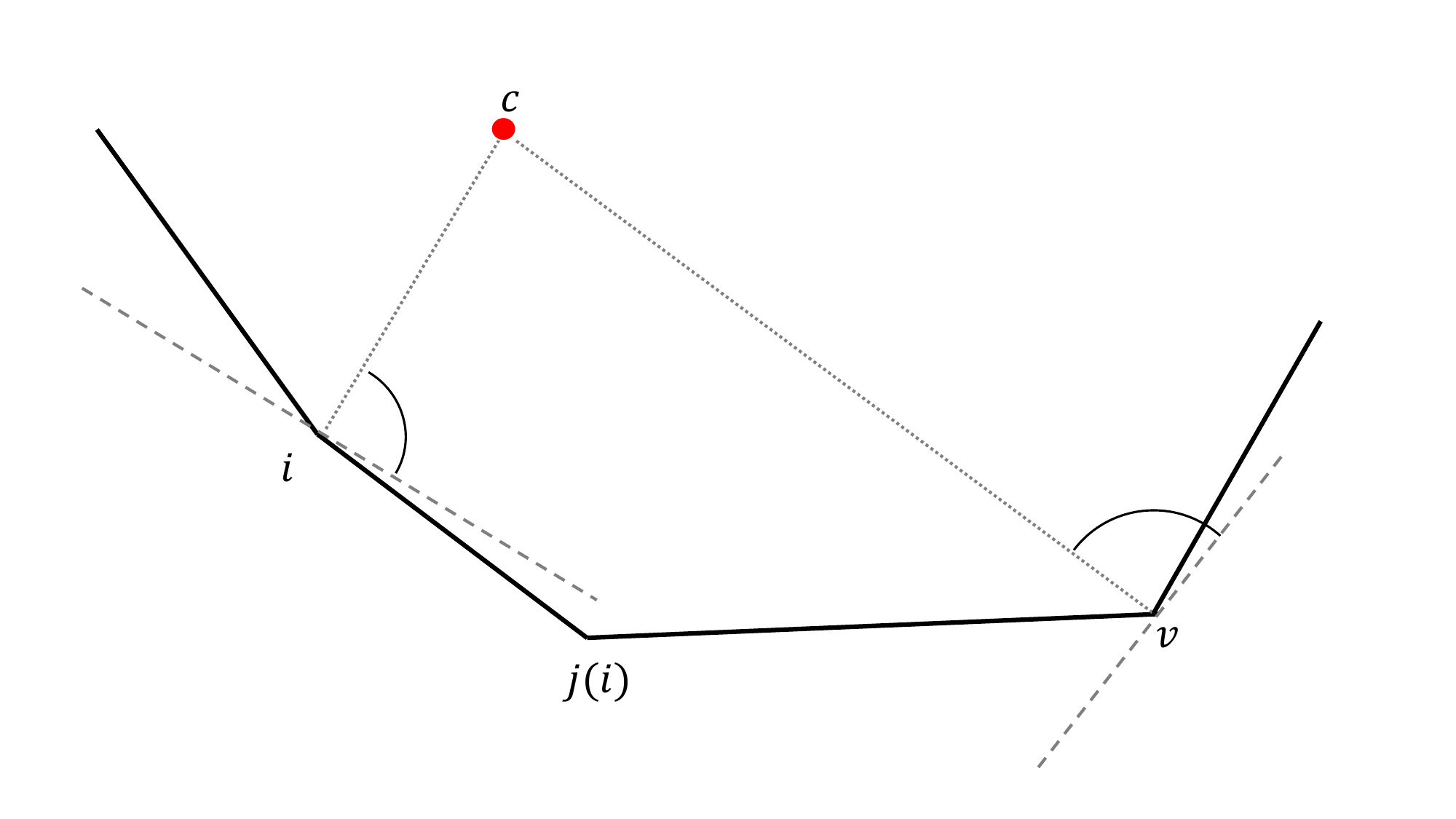}
    \caption{Illustration of Definition \ref{def:unstable} and Lemma \ref{lem:not-unstable-in-the-shadow}. The point $c$ is the center of mass. The vertex $i$ is in the shadow of the adjacent vertex $j(i)$, and is thus not unstable. Vertex $v$ is unstable.}
    \label{fig:unstable}
\end{figure}

%\begin{lemma}
%No unstable vertex of a polyhedron is in the shadow of any other vertex.
%\end{lemma}

%\begin{lemma}
%If the vertex $i$ of a polyhedron is not in the shadow of any \emph{adjacent} vertex, then $i$ is an unstable vertex.
%\end{lemma}

\begin{lemma}\label{lem:not-unstable-in-the-shadow}
If the vertex of a polyhedron is not unstable, then it is in the shadow of an adjacent vertex.
\end{lemma}
\begin{proof}
Let $i$ be a vertex that is not unstable, let $j_1,\dots,j_k$ denote its adjacent vertices, and denote by $H$ the closed half-space whose boundary is the vertex-orthoplane at $i$ and which does not contain the center of mass of the polyhedron. We need to show that $H$ contains a vertex of the polyhedron adjacent to $i$.

By Definition \ref{def:unstable}, the polyhedron has a point different from $i$ contained in $H$, and by extension, so does the (unbounded) polyhedron \[i + \operatorname{cone}(\operatorname{span}(\{j_1-i, \dots, j_k-i\})).\]
Therefore, at least one of the halflines extended from $i$ towards an adjacent vertex is contained in $H$. This halfline contains the vertex $j$ adjacent to $i$ in $H$.
\end{proof}

%\begin{corollary}\label{cor:main}
%Let $P$ be a mono-unstable polyhedron whose center of mass is at the origin. Let its vertices be labeled $1, \dots, V$, with $v$ being its only unstable vertex, and let $\vr_i$ be the coordinate vector of vertex $i$ for each $i=1,\dots,V$. Then for every vertex $i \in \{1,\dots,V\}\setminus\{v\}$ there exists a vertex $j(i)$ adjacent to $i$ satisfying
%\begin{equation}\label{eq:shadow-ij}
%(\vr_i-\vr_{j(i)})^\T\vr_i \leq 0.
%\end{equation}
%\end{corollary}

%Thus, in order to show that there exists no mono-unstable polyhedra with $V$ vertices, it suffices to show that the system of inequalities \eqref{eq:shadow-ij} has no solution $(\vr_1,\dots,\vr_V)$ in which the vertices $\vr_i$ determine a polyhedron whose center of mass is at the origin, for any choice of $j(i)$, $i \in \{1,\dots,V\}\setminus\{v\}$. The following observations help reduce the number of maps $j$ to consider.

%As noted, every vertex that is \emph{not} unstable is shadowed by one if its neighbors.
Using the next Lemma, we can further refine the shadowing relation on the vertices of a polyhedron.

\begin{lemma}%[\protect{\cite[Lemma 2]{IJSS}}]
\label{lem:monotone}
If vertex $i$ is in the shadow of vertex $j$, then $\|\vr_i - \vc\| < \|\vr_j - \vc\|$.
\end{lemma}
\begin{proof}
Since both the assumption and the conclusion are invariant to a translation of the polyhedron, we may assume without loss of generality that the center of mass is at $\vc=\vzero$. Then \eqref{eq:shadow} simplifies to $\|\vr_i\|^2 \leq \vr_j^\T \vr_i$, from which we get the desired inequality
\[ 0 < \|\vr_i - \vr_j\|^2 \overset{\eqref{eq:shadow}}{\leq} \|\vr_j\|^2 - \|\vr_i\|^2. \qedhere\]
\end{proof}

An immediate consequence of Lemma \ref{lem:monotone} is that the shadowing relation has \emph{no cycles}. Together with Lemma \ref{lem:not-unstable-in-the-shadow}, this also implies the well-known fact that the unstable vertex of a mono-unstable polyhedron is the vertex that is the farthest from the center of mass.

Consider now the (directed) graph $SG_P$ whose vertices are the vertices $1,\dots,V$ of the polyhedron $P$ and whose arcs are
\[ \{(i,j)\,|\, i \textrm{ is shadowed by the adjacent vertex } j\}. \]
We may call this the \emph{shadowing graph} of the polyhedron. Based on the above, for every polyhedron $P$, the shadowing graph $SG_P$ has no directed cycles, and (ignoring the directions) it is a subgraph of the graph formed by the vertices and edges of the polyhedron (commonly known as \emph{the graph of $P$}). If, in addition, $P$ is mono-unstable, then every vertex of $SG_P$ except for the unstable vertex has an outgoing arc in this graph by Lemma \ref{lem:not-unstable-in-the-shadow}, and thus, $SG_P$ is a (not necessarily disjoint) union of directed paths connecting every vertex to the unstable one.

In summary, for mono-unstable $P$, the shadowing graph $SG_P$ is a subgraph of the graph of $P$, and it contains as a subgraph a \emph{rooted tree} \cite[p.~187]{Harary1969} with the unstable vertex as its root. Except for this root, every vertex is shadowed by its parent vertex in this rooted tree.  %In particular, the relation $i \mapsto j(i)$ in Corollary \ref{cor:main} corresponds to a rooted subtree of the graph of $P$.

The observations of this section are summarized in the following theorem.
\begin{theorem}\label{thm:shadow}
Let $P$ be a mono-unstable polyhedron whose center of mass is at the origin. Let its vertices be labeled $1, \dots, V$, with $v$ being its only unstable vertex, and let $\vr_i$ be the coordinate vector of vertex $i$ for each $i=1,\dots,V$. Then for every vertex $i \in \{1,\dots,V\}\setminus\{v\}$ there exists a (not necessarily unique) vertex $j(i)$ adjacent to $i$ satisfying
\begin{equation}\label{eq:shadow-ij}
(\vr_i-\vr_{j(i)})^\T\vr_i \leq 0.
\end{equation}
The (undirected) graph with vertices $\{1,\dots,V\}$ and edges $\{ \{i,j(i)\}\,|\, i\neq v \}$ is a rooted subtree of the graph of $P$.
\end{theorem}

Theorem \ref{thm:shadow} can be interpreted as follows. If we can prove that for a given number of vertices $V$ the system of inequalities \eqref{eq:shadow-ij} does not have a solution, no matter how we assign a shadowing vertex $j(i)$ to each vertex $i$, then no mono-unstable polyhedron with $V$ vertices exists. Thus, in principle, we can accomplish our goal of proving the non-existence of mono-unstable polyhedra with $V$ vertices by enumerating all possible polyhedral graphs and all rooted subtrees of all such graphs, and showing that the system \eqref{eq:shadow-ij} is inconsistent with the center of mass being at the origin. Since $\vr_i=\vzero, i=1,\dots,V$ is a (non-geometric) solution that meets both criteria, we need to add at least one more condition, for example, that the unstable vertex $v$ is at the point $(1,0,0)$, which we can assume without loss of generality.

This conceptual algorithm is summarized in Algorithm \ref{alg:basic}. The notation $\textrm{CoM}(\vr_1,\dots,\vr_V)$ in \eqref{eq:com-eq} stands for the center of mass of the polyhedron whose vertices are $\vr_1,\dots,\vr_V$.

\begin{algorithm2e}
    \caption{Sufficient conceptual method for proving the non-existence of mono-unstable polyhedra on $V$ vertices}\label{alg:basic}
    \DontPrintSemicolon
    %\SetKwInOut{Input}{input}
    %\SetKwInOut{Output}{output}
    
    %\Input{An integer $V>4$}
    %\Output{A certificate for the nonexistence....}
    \For{each polyhedral graph $G$ on $V$ vertices}{
        \For{each vertex $v$ of $G$}{
            Enumerate every cycle-free mapping $j:\{1,\dots,V\}\setminus\{v\} \to \{1,\dots,V\}$ for which $j(i)$ is adjacent to $i$.
            
            \For{each such map $j$}{
            Prove that the following system of is inconsistent
            \begin{subequations}\label{eq:system-alg1}
            \begin{align}
                \vr_v = (1,0,0),\\            
                (\vr_i-\vr_{j(i)})^\T\vr_i &\leq 0 & i\in\{1,\dots,V\}\setminus\{v\}, \label{eq:shadow-alg1}\\
                \textrm{CoM}(\vr_1,\dots,\vr_V) &= \vzero. \label{eq:com-eq}
            \end{align}
            \end{subequations}
            }
        }
    }
\end{algorithm2e}

The equation \eqref{eq:com-eq} can be reformulated as a system of polynomial equations of degree four (see below) in the variables $\vr_i$. Therefore, \eqref{eq:system-alg1} is a system of polynomial equations and inequalities. Our task is to certify the infeasibility of these systems.

\section{Infeasibility certificates for systems of polynomial inequalites}

The infeasibility of systems of inequalities can be decided and certified easily when the inequalities are all \emph{linear}: the celebrated Farkas' Lemma states that the linear system $\vA\vx\leq\vb$ with $\vA\in\R^{m \times n}$ is infeasible if and only if there exists a vector $\vy\in\R_+^m$ satisfying $\vA^\T\vy=\vzero$ and $\vb^\T\vy = -1$. In other words, it is infeasible if and only if there are nonnegative coefficients $y_i (i=1,\dots, m)$ such that the corresponding linear combination of the inequalities is the trivially infeasible inequality $0 \leq -1$ \cite{Farkas1902}. If found, such a vector $\vy$ serves as an easily verifiable \emph{certificate of infeasibility} for the system, particularly when the components of $\vy$ are ``small'' rational numbers. Such a rational vector $\vy$, with bit size polynomial in the bit size of $\vA$ and $\vb$, can easily be found using linear programming \cite{Chvatal1983}.

In \cite{Bozoki}, the authors used a similar idea for certifying the infeasibility of systems of \emph{quadratic} inequalities. Although exact analogs of Farkas' Lemma are not known in the quadratic case (and likely do not exist on computational complexity grounds, since recognizing infeasible systems of quadratic inequalities is NP-hard), an analogous sufficient condition can be formulated for the infeasibility of a system
\begin{equation}
\label{eq:generic-quad-system}
a_i + \vb_i^\T\vx + \vx^\T\vQ_i\vx \leq 0 \quad (i=1,\dots,m),
\end{equation}
as follows: suppose there exist nonnegative coefficients $y_1,\dots,y_m$ such that $\sum_{i=1}^m y_i \vQ_i$ is positive definite and that the function value at the unique minimizer of the function
\begin{equation}
\label{eq:f}
\vx \mapsto \sum_{i=1}^m y_i(a_i + \vb_i^\T\vx + \vx^\T\vQ_i\vx)
\end{equation}
is strictly positive. (Both of these conditions can be easily verified for a given $\vy\in\R^m$.) Then it is immediately clear that \eqref{eq:generic-quad-system} does not have a solution. Similarly to the linear case, the vector of coefficients $\vy$ is an easily verifiable certificate of infeasibility. In \cite{PappEJOR}, the authors also showed that such a certificate (when exists) can be computed efficiently using \emph{semidefinite programming} \cite{VandenbergheBoyd1996} and used this technique to show that the the smallest mono-unstable convex polyhedron with point masses has 11 vertices by certifying the infeasibility of $362\,880$ systems of homogeneous quadratic inequalities. The following theorem is an adaptation of \cite[Corollary 9]{PappEJOR} to systems of the form \eqref{eq:generic-quad-system}, which is what we will use in our final algorithm in Section \ref{sec:refining}.

\begin{theorem}\label{thm:SDP}
Let a system of quadratic inequalities \eqref{eq:generic-quad-system} be given, and let $\hat{\vQ}_i$ denote the associated coefficient matrices
\[\hat{\vQ}_i := \begin{pmatrix}a_i & \vb_i^\T/2 \\ \vb_i/2 & \vQ_i\end{pmatrix}.\]
Consider the following semidefinite optimization problem:
\begin{equation}\label{eq:SDP}
\begin{aligned}
&\maximize_{(z,\vy)\in\R \times \R^m}\quad    && z\\
&\st &&\sum_{i=1}^m \hat{\vQ}_i y_i  \succcurlyeq z \vI\\
&&& \|\vy\|_2 \leq 1\\
&&& y_i \geq z \qquad i=1,\dots,m.
\end{aligned}
\end{equation}
If the optimal value of \eqref{eq:SDP} is positive, then \eqref{eq:generic-quad-system} is infeasible. Moreover, every rational feasible solution $(z,\vy)$ of \eqref{eq:SDP} with $z>0$ yields a certificate of infeasibility $\vy$ for \eqref{eq:generic-quad-system}.
\end{theorem}
\begin{proof}
The optimization problem \eqref{eq:SDP} has an optimal solution, because the objective function is continuous and the feasible region is non-empty (consider $(z,\vy) = (0,\vzero)$), bounded in $\vy$, and bounded from above in $z$.

Consider a feasible solution $(z,\vy)$ with $z>0$, and let $f$ denote the corresponding quadratic polynomial \eqref{eq:f}. From the feasibility of $(z,\vy)$ and $z>0$, we see that each $y_i > 0$ and that $\sum_{i=1}^m \hat\vQ_i y_i $ is positive definite, therefore so is its lower-right block $\sum_{i=1}^m \vQ_i y_i $, which is the Hessian of $\frac{1}{2}f$. That is, $f$ is a strictly convex quadratic polynomial, with a unique minimizer.

This unique minimizer is the point $\tilde\vx$ where $\nabla f(\tilde\vx) = \vzero$, given by
\[ \tilde\vx = -\frac{1}{2}\left(\sum_{i=1}^m \vQ_i y_i\right)^{-1}\left(\sum_{i=1}^m \vb_i y_i\right), \]
and thus the minimum value of $f$ is
\begin{align*}
f(\tilde\vx) &= \left(\sum_{i=1}^m a_i y_i\right) + \left(\sum_{i=1}^m \vb_i y_i\right)\tilde\vx + {\tilde\vx}^\T\left(\sum_{i=1}^m \vQ_i y_i\right)\tilde\vx\\
&= \left(\sum_{i=1}^m a_i y_i\right) - 
\left(\sum_{i=1}^m \frac{1}{2}\vb_i y_i\right)^\T \left(\sum_{i=1}^m \vQ_i y_i\right)^{-1} \left(\sum_{i=1}^m \frac{1}{2}\vb_i y_i\right)
\end{align*}
The quantity on the right-hand side of this equation is precisely the Schur complement of the lower-right block of the matrix $\sum_{i=1}^m \hat\vQ_i y_i$. Since both $\sum_{i=1}^m \hat\vQ_i y_i$ and $\sum_{i=1}^m \vQ_i y_i$ are positive definite, the Schur Complement Lemma \cite[Theorem 1.12]{Zhang2005} tells us that the Schur complement $f(\tilde\vx)$ is also positive. Therefore, the polynomial $f$ is strictly positive everywhere, which is impossible if $\vy\geq 0$ and \eqref{eq:generic-quad-system} is feasible.
\end{proof}

The important consequence of Theorem \ref{thm:SDP} and its proof is that we need not find the exact optimal solution of the problem \eqref{eq:SDP}; in fact, we need not even find an exact feasible solution. As long as our numerical solution of \eqref{eq:SDP} is ``close enough to feasibility'' that $\vy\geq 0$ and $\sum_{i=1}^m\hat\vQ_iy_i$ is positive definite, the floating-point vector $\vy$ is automatically a rational vector that serves as an easily verifiable certificate proving that \eqref{eq:generic-quad-system} is infeasible.

The idea does not easily generalize to polynomials of degree greater than two, such as the system \eqref{eq:system-alg1}, because for polynomials of degree 4 or more, even the complexity of recognizing convex polynomials is NP-hard. More complex sufficient conditions of global positivity applicable to polynomials of any degree involve \emph{sum-of-squares} certificates \cite{Blekherman2012,Laurent2009}, which can also be computed using semidefinite programming \cite{DurRendl2021,Lasserre2001}, but the computation of these certificates is far too computationally demanding for our task. Instead, in our final algorithm (Algorithm \ref{alg:complete}) we will replace the system \eqref{eq:system-alg1} with a quadratic one whose infeasibility implies the infeasibility of \eqref{eq:system-alg1}, and apply Theorem \ref{thm:SDP} above to the resulting systems. This is detailed in Section \ref{sec:tetrahedral-decompositions}.

\section{Refining the algorithm}
\label{sec:refining}

The conceptual algorithm, as shown in Algorithm \ref{alg:basic}, is impractical for multiple reasons. Most importantly, the center-of-mass function $\textrm{CoM}$, a dense $3V$-variate polynomial of degree $4$, is far too unwieldy for existing methods to compute infeasibility certificates even for a single instance of the system \eqref{eq:system-alg1}. Additionally, the triple for-loop creates a large number of cases to consider (that is, a large number of systems of inequalities whose infeasibility must be certified) even for modest values of $V$. In this section, we investigate how to reduce the number of cases in each for-loop and how to simplify the system \eqref{eq:system-alg1} to one that is easier to handle.

The complete algorithm is shown in Algorithm \ref{alg:complete}; in the rest of this section we explain the details. If Algorithm \ref{alg:complete} successfully certifies the infeasibility of each generated polynomial system, then these certificates serve as a computer-generated proof of the nonexistence of mono-unstable 3-dimensional polyhedra with $V$ vertices.

\subsection{Reducing the number of cases}
\phantom{0}\newline
\noindent\textit{Triangulated graphs.} The vertices of every convex polyhedron can be perturbed (by an arbitrarily small positive amount) in such a way that each vertex remains a vertex, and that no four vertices lie in the same plane. Such a perturbed polyhedron is simplicial (has only triangular faces), and if the original polyhedron was mono-unstable, then (for sufficiently small perturbations) so is the perturbed one. Hence, in the outermost loop of the algorithm, it is sufficient to consider only simplicial polyhedra, which correspond to maximal (3-connected) planar graphs. 

For example, in the case of $V=5$, there is only one maximal planar graph, the triangular bipyramid.

\smallskip

\noindent\textit{Eliminating symmetries.} It is clear that for some polyhedral graphs we need not consider each vertex to be a candidate for being the only unstable vertex. The automorphically equivalent vertices of a graph can be determined using graph automorphism algorithms such as those in \cite{McKayPiperno2014}, or manually for the smallest values of $V$. If $u$ is a vertex of a polyhedron that cannot be the only unstable vertex, and the graph of the polyhedron has an automorphism mapping $u$ to $v$, then $v$ also cannot be the only unstable vertex.

For example, the triangular bipyramid ($V=5$) has only two automorphically inequivalent sets of vertices: the vertices with the same degree are symmetric. See also Figure \ref{fig:maximal-simplicial-graphs}, the graph of the polyhedron has an automorphism mapping vertex $1$ to vertex $2$, and another two automorphisms mapping vertex $3$ to $4$ and $5$, respectively. Therefore, it is sufficient to show that vertices $1$ and $3$ cannot be the only unstable vertex of the polyhedron, and it follows that no triangular bipyramid can be mono-unstable.

Additional symmetries could be exploited (for example, certain rooted subtrees lead to isomorphic systems of infeasible inequalities), but because this simplification is only needed to reduce computational time, which for low values of $V$ was already minimal after eliminating vertex symmetries, we did not pursue more elaborate techniques.

\subsection{Tetrahedral decompositions}
\label{sec:tetrahedral-decompositions}

The infeasibility of the systems \eqref{eq:system-alg1} is challenging to prove primarily because of the last set of equations \eqref{eq:com-eq}. We can replace these equations with simpler quadratic inequalities that are necessary (but not sufficient) for the center of mass to be at zero using a suitably chosen 3-triangulation, as follows.

We continue to assume without loss of generality that the coordinate vector of the unstable vertex $v$ is
\begin{equation}\label{eq:rv_eq_0}
\vr_v = (1,0,0).
\end{equation}
Now, consider those tetrahedra whose vertices are $v$ and the three vertices of a face of $P$ that does not contain $v$. These tetrahedra determine a geometric partition of $P$: their union is $P$, but no two of them have a common interior point. (Figure \ref{fig:tetra-decomp}.) We shall refer to this collection of tetrahedra as the \emph{tetrahedral decomposition of $P$ corresponding to $v$}. Alternatively, since the decomposition depends solely on the graph $G$ of the polyhedron, we may refer to it as the decomposition of $G$ rather than $P$.

\begin{figure}
    \centering
    \includegraphics[width=0.8\linewidth]{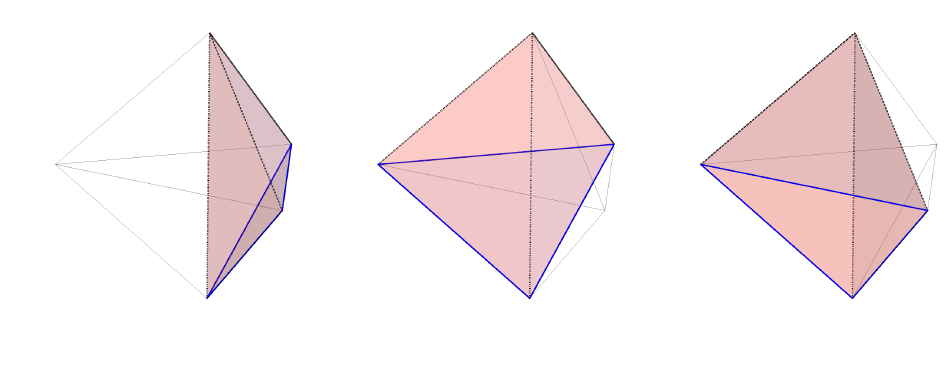}
    \caption{A tetrahedral decomposition of the triangular bipyramid. Shaded are the three tetrahedra that form the decomposition with respect to the top vertex. Each tetrahedron corresponds to a triangular face (thick lines) not containing this vertex.}
    \label{fig:tetra-decomp}
\end{figure}

%Let us consider a convex polyhedron and assume that vertex $\mathbf{r}_1$ is the only unstable equilibrium of the polyhedron. Consider the 3-triangulation of the polyhedron as follows. Triangulate the faces with at least 4 vertices arbitrarily and add these face diagonals to the set of edges of the polyhedral graph of the polyhedron. Then every face of the extended polyhedral graph is triangular. Figures XXX and YYY show the possible polyhedral graphs on $5$ and $6$ vertices, as well as the triangular ones among them. Consider all the tetrahedra generated by vertex $\mathbf{r}_1$ and the polyhedron's triangular faces, disjoint to vertex $\mathbf{r}_1$. 
Having the polyhedron thus decomposed, its center of mass, $(0,0,0)$ can be written as a weighted sum of the tetrahedra's center of mass, where the weights are the volumes of the tetrahedra. Here we use the equation of the first coordinates only:
\[
 0  =  \sum\limits_{\ell=1}^{N} \text{vol}(T_\ell) x_\ell,
\]
where $N$ denotes the number of tetrahedra, vol($T_\ell$) denotes the volume of the $\ell$-th tetrahedron, and $x_\ell$ denotes the first coordinate of the $\ell$-th tetrahedron's center of mass.
Obviously vol($T_\ell$) $> 0$ for all $\ell=1,2,\ldots,N$, and 
\[
%x_k = \frac{r_{v,1}+r_{k_1,1}+r_{k_2,1}+r_{k_3,1}}{4},
x_\ell = \frac{1}{4}\sum_{k \in V(T_\ell)} r_{k,1}
\]
wherein $V(T_\ell)$ denotes the set of vertices of $T_\ell$.

It can be readily seen that there is at least one tetrahedron $T$ in the tetrahedral decomposition whose center of mass has a nonpositive first component, that is,
\begin{equation}
\label{eq:Sanyi}
\sum_{k\in V(T)} r_{k,1} \leq 0.
\end{equation}
Now, our approach is to replace the system of equations \eqref{eq:com-eq} with an instance of \eqref{eq:Sanyi} for a single tetrahedron $T$ in the decomposition, along with \eqref{eq:rv_eq_0}. This simpler necessary condition involves only linear equations and inequalities instead of the complicated expressions characterizing the exact location of the center of mass. If the resulting quadratic system has no solution for any of the tetrahedra, then the non-existence of mono-unstable polyhedra is proven.

The complete algorithm that incorporates each of the refinements discussed in this section is shown below as Algorithm \ref{alg:complete}. Two small clarifications are in order: first, the equation \eqref{eq:v100} is not explicitly included in the system; rather, the vector $\vr_v$ is eliminated from the remaining inequalities by this substitution; thus, \eqref{eq:system-alg2} is an inhomogeneous system of linear and quadratic inequalities. Second, the subscript $j(i)$ in \eqref{eq:shadow-alg2} refers to the shadowing relation $j$ discussed earlier, corresponding to selected the rooted subtree of $G$.
Inequalities \eqref{eq:shadow-alg2} are the same as \eqref{eq:shadow-alg1}; so are \eqref{eq:Sanyi-alg2}
and \eqref{eq:Sanyi}.

\begin{algorithm2e}
    \caption{The implemented sufficient method that proves the non-existence of homogeneous, mono-unstable polyhedra on $V \leq 6$ vertices}\label{alg:complete}
    \DontPrintSemicolon
    %\SetKwInOut{Input}{input}
    %\SetKwInOut{Output}{output}
    
    %\Input{An integer $V>4$}
    %\Output{A certificate for the nonexistence....}
    \For{every maximal planar graph $G$ on $V$ vertices}{
        Determine the automorphism classes of $G$ and choose a representative vertex from each class.
        
        \For{each representative vertex $v$}{
            Compute the tetrahedral decomposition of $G$ corresponding to $v$.\\
            Compute each rooted subtree of $G$ rooted at $v$.\\
            \For{each rooted subtree and tetrahedron $T$ in the decomposition}{
            Find an infeasibility certificate for the quadratic system
            \begin{subequations}\label{eq:system-alg2}
            \begin{align}
                &\vr_v = (1,0,0), \label{eq:v100}\\
                &(\vr_i-\vr_{j(i)})^\T\vr_i \leq 0 & i\in\{1,\dots,V\}\setminus\{v\}, \label{eq:shadow-alg2}\\
                &\sum_{k\in V(T)} r_{k,1} \leq 0.  \label{eq:Sanyi-alg2}
            \end{align}
            \end{subequations}
            by solving the associated semidefinite program \eqref{eq:SDP}.
            
%            Solve the appropriate SDP to compute an infeasibility certificate for the system quadratic inequalities \eqref{?}, which consists of:
%            \begin{itemize}
%                \item Each vertex except $v$ is shadowed by its ancestor in $T$.
%                \item The coordinate vector of $v$ is $(1,0,0)$.
%                \item The set of quadratic inequalities from the tetrahedral decomposition that is sufficient to ensure that the first coordinate of the center of mass is non-positive. RESZLETESEBBEN / FORMALISABBAN.
%            \end{itemize}
            }
        }
    }
\end{algorithm2e}

\section{Computational results: there are no homogeneous mono-unstable convex polyhedra with fewer than 7 vertices}\label{sec:results}

We have implemented Algorithm \ref{alg:complete} in the computer algebra system \emph{Mathematica}, in which the combinatorial components of the algorithm are either readily available or are easily implemented, and which also incorporates the semidefinite programming solver CSDP, which can be used to compute the infeasibility certificates. The results for $V \in \{5,6\}$ are detailed below and summarized in Table \ref{tbl:results}; computation times were measured on a MacBook Pro laptop with an Apple M2 Pro CPU.

The list of the computed rational feasible solutions of \eqref{eq:SDP} certifying the infeasibility of the
systems generated by Algorithm \ref{alg:complete} can be found in the public repository \url{https://github.com/dpapp-github/mono-unstable/}.
This, along with the proof of Theorem \ref{thm:SDP}, serves as the independently verifiable computer-assisted proof of Theorem \ref{thm:main}.

\begin{figure}
    \centering
    \includegraphics[width=1\linewidth]{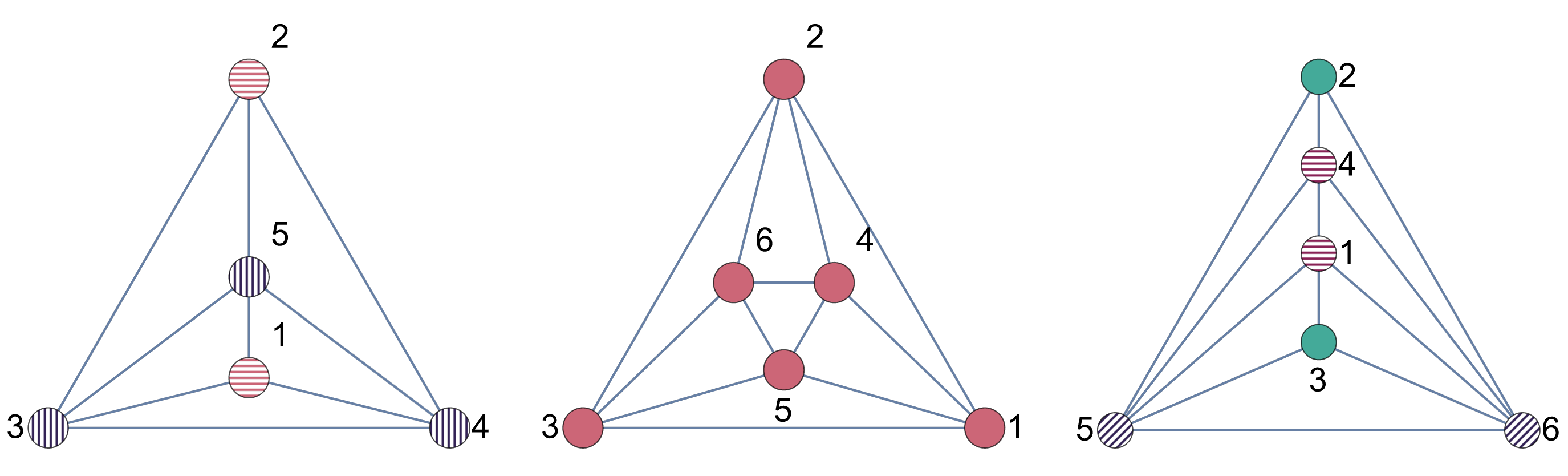}
    \caption{Maximal simplicial graphs with $5$ or $6$ vertices. Left to right: the triangular bipyramid, the regular octahedral graph, and an ``irregular'' octahedron. Colors/patterns show the symmetric vertices.}
    \label{fig:maximal-simplicial-graphs}
\end{figure}

\paragraph{$V=5$:} There is only one maximal planar graph $G$ on $5$ vertices, the triangular bipyramid. (Figure \ref{fig:maximal-simplicial-graphs}.) $G$ has 75 subtrees (and therefore, 75 rooted trees with the root at any fixed vertex). Up to symmetry, it has two types of vertices that are candidates to be the unstable vertex $v$, as vertices of the same degree are symmetric. If $v$ is chosen to be a degree-$3$ vertex, then the tetrahedral decomposition consists of $3$ tetrahedra; if $v$ has degree $4$, then the decomposition consists of only $2$. Therefore, the total number of systems of inequalities \eqref{eq:system-alg2} that we need to certify as infeasible is $75(3+2)=375$; each system has $5$ inequalities. Each were certified to be infeasible; the optimal value of \eqref{eq:SDP} was greater than $0.068$ in each instance. The total computation time was approximately 1 second.

\paragraph{$V=6$:} There are two maximal planar graphs $G$ on $6$ vertices:  the regular octahedral graph, and the one that to the best of our knowledge does not have a commonly used name, but we shall refer to as the ``irregular octahedron'', as it is the only triangulated polyhedron with eight faces other than the regular octahedron---it is the first graph in the catalog \cite{BrittonDunitz1973} of polyhedral graphs, and the last one on \mbox{Figure \ref{fig:maximal-simplicial-graphs}}.

The regular octachedron has $384$ subtrees; its vertices are all symmetric, and their corresponding tetrahedral decompositions consist of 4 tetrahedra. The irregular one has $336$ subtrees and $3$ classes of vertices; the corresponding tetrahedral decompositions have $3$, $4$, and $5$ tetrahedra respectively.  Therefore, the total number of systems of inequalities \eqref{eq:system-alg2} that we need to certify as infeasible is $384 \cdot 4 + 336(3+4+5) = 5568$; each system has $6$ inequalities. Each were certified to be infeasible; the optimal value of \eqref{eq:SDP} was greater than $0.077$ in each instance. The total computation time was approximately 19 seconds.

\begin{table}[tb]
\caption{Summary of the computations showing the number of cases at the various levels of case analysis detailed in Section \ref{sec:results}. Notation: $V$ is the number of vertices; $G$ is the graph of the polyhedron; $v$ is the label of the candidate unstable vertex on Figure~\ref{fig:maximal-simplicial-graphs}. The third column shows the number of subtrees of $G$; the fifth one the number of tetrahedra in the tetrahedral decomposition of $G$ corresponding to $v$. The total number of systems whose infeasibility needs to be certified is $75(3+2) + 384\cdot 4 + 336(4+5+3) = 5943$.\\}
\label{tbl:results}
\centering
\begin{tabular}{ccccc}
\toprule
$V$ & $G$ & \# trees & $v$ & \# tetrahedra\\
\midrule
\multirow{2}{*}{5} & \multirow{2}{12ex}{triangular bipyramid} & \multirow{2}{*}{75} & 1 & 3 \\
& & & 3 & 2\\
\midrule
\multirow{5}{*}{6} & \multirow{2}{12ex}{regular octahedron} & \multirow{2}{*}{384} & \multirow{2}{*}{1} & \multirow{2}{*}{4}\\
\\
                   \cmidrule(rl){2-5}
                   & \multirow{3}{12ex}{``irregular'' octahedron} & \multirow{3}{*}{336} & 1 & 4\\
                   &                        &                      & 2 & 5\\
                   &                        &                      & 5 & 3\\
\bottomrule
\end{tabular}
\end{table}

\section{Discussion}

Although all of the arguments presented are valid for $V>6$, they are not sufficient to prove the non-existence of homogeneous mono-unstable polyhedra with 7 vertices. Algorithm \ref{alg:complete} runs within minutes for $V=7$, but it only certifies the infeasibility of approximately 99\% of the generated systems of inequalities, and returns the optimal value $z=0$ for the remaining few cases. In principle, there are four possibilities:
\begin{enumerate}
    \item A homogeneous mono-unstable polyhedra with 7 vertices exists. We conjecture this to be highly unlikely.
    \item All of the systems \eqref{eq:system-alg2} are infeasible, but Theorem \ref{thm:SDP}, a sufficient but not necessary condition, is not strong enough to prove this.
    \item Some of the systems \eqref{eq:system-alg2} have solutions, but they are not geometric, that is, they do not correspond to vertices of a polyhedron.
    \item Some of the systems \eqref{eq:system-alg2} have solutions that correspond to vertices of a polyhedron, but they are nevertheless not a solution to the original problem. For example, the center of mass might fail to be at the origin.
\end{enumerate}

The last option could be avoided by reverting back to the conceptual \mbox{Algorithm \ref{alg:basic}}, and the penultimate case could also be prevented in principle by further tightening the formulation \eqref{eq:system-alg1} to exclude non-polyhedral solutions. However, these changes require far more complex systems, involving a larger number of inequalities of higher degrees, with substantially more variables, making the certification of infeasibility considerably more difficult.

As for replacing Theorem \ref{thm:SDP} with something more powerful: conceptually, we could avoid resorting to sufficient conditions altogether, since the problem of recognizing infeasible systems of polynomial inequalities is algorithmically decidable \cite{Tarski1951,Renegar1992}. Moreover, the existence of independently verifiable certificates of infeasibility are provided by the various \emph{Positivstellensatz} theorems of algebraic geometry, such as Putinar's Positivstellensatz \cite{Putinar1993}, and these certificates can in principle be computed using semidefinite programming \cite{Lasserre2001}, even in rational arithmetic \cite{DavisPapp2024}. However, in our experience, the current computational tools of semidefinite programming and sums-of-squares optimization are not sufficient to resolve problems of this scale.
\bibliographystyle{amsplain}
\bibliography{homogen}
\end{document}